\documentclass[reqno]{amsart}
\usepackage{hyperref}

\def\N{{\mathbb{N}}}

\def\R{{\mathbb{R}}}

\def\1{{\mathbb{1}}}

\usepackage{amssymb}
\usepackage{xcolor}
\DeclareUnicodeCharacter{2264}{\leq}
\DeclareUnicodeCharacter{2266}{\leqq}

\theoremstyle{plain}
\newtheorem{theorem}{Theorem}
\newtheorem{proposition}{Proposition}
\newtheorem{definition}{Definition}
\newtheorem{lemma}{Lemma} 
\newtheorem{corollary}{Corollary}
\theoremstyle{remark}
\newtheorem{remark}{Remark}
\newtheorem{example}{Example}

\title[Alternative theorem and applications]{Alternative theorem for sequences of functions and applications to optimisation}
\author{Mohammed Bachir, Rongzhen Lyu}

\begin{document}

\date{\today} 
\subjclass{}
\address{Laboratoire SAMM 4543, Universit\'e Paris 1 Panth\'eon-Sorbonne, France}

\email{Mohammed.Bachir@univ-paris1.fr}
\email{Rongzhen.Lyu@etu.univ-paris1.fr}

\begin{abstract}
We present a new alternative theorems for sequences of functions. As applications, we extend recent results in the literature related to first-order necessary conditions for optimality problems. Our contributions involve extending well-known results, previously established for a finite number of inequality constraints to a countable number of inequality constraints. This extension is achieved using the Dini differentiability concept which is more general than Fréchet or Gâteaux differentiability. We will illustrate our results by giving examples of optimisation problems with a finite or countable number of inequality constraints where the functions are not Gateaux-differentiable but only upper Dini-differentiable.
\end{abstract}

\maketitle
{\bf Keywords:} Alternative theorem; Lagrange multiplier rule; Upper Dini derivative; Optimization;  Countable inequality constraints; Infinite dimension.

{\bf 2020 Mathematics Subject Classiﬁcation:} \subjclass{Primary  90C30, 49K99, 46N10 Secondary 90C26, 90C48} 
\tableofcontents

\vskip5mm

\section{Introduction} \label{S1}

There is a vast body of literature on alternative theorems, with numerous studies addressing alternatives in both convex and non-convex frameworks and their applications to optimization theory. A comprehensive survey on this topic by Dinh and Jeyakuma can be found in \cite{DJ}, along with many references (see also \cite{Cr1, Cr2, Je1, JON, Ld, RG0, RG1}, although this list is not exhaustive). The earliest version of a nonlinear alternative theorem was introduced by Fan, Glicksberg, and Hoffman \cite{FGH}, and is stated as follows:  Let $C$ be a convex set of some real vector space and $f_i: C \to \R$, $i=1,..., n$ be convex functions. Then, either the system $x\in C$, $f_i(x) <0, i=1,...,n$ has a solution or there exists $\lambda^*\in \R_+^n$, $\lambda^*\neq 0$ such that $\sum_{i=1}^n \lambda^*_i f_i(x)\geq 0$ for all $x\in C$. 
\vskip5mm
The aim of this article is to present new versions of the alternative theorem and to derive applications in both convex and non-convex optimization theory, generalizing recent results by Blot in \cite{Bl}, Yilmaz in \cite{Yi}, Bachir and Lyu in \cite{BL} and the old result  of Pourciau in \cite{Po}, from a finite number of inequality constraints to a countable number of inequality constraints and by using only the upper Dini differentiability which is more general than the Gateaux differentiability. Our main results are Theorem ~\ref{thm4} and Theorem ~\ref{thm} and their corollaries.
\vskip5mm
As a consequence of Theorem \ref{thm4}, we obtain an alternative theorem that applies to invex Lipschitz functions (Corollary \ref{cor3}). Invexity is a concept that generalizes convexity and was introduced by Hanson \cite{Ha} (see Section \ref{invex} for further details). Theorem \ref{thm} gives a direct extension of the result of Fan, Glicksberg, and Hoffman in \cite{FGH}, where a finite familly of convex functions is replaced by a sequence of convex functions pointwize converging to some function.  We then apply our results to optimization problems in the context of Lagrange multipliers. Specifically, we consider the following problem: Let $E$ be a real normed vector space, $\Omega$ be a nonempty open subset of $E$, and $(f_n)$ be a sequence of real-valued functions on $\Omega$,
\begin{equation*}
(\mathcal{P}_\infty)
\left \{
\begin{array}
[c]{l}
\min f_0\\
x\in \Omega\\
\forall n\geq 1, f_{n}(x) \leq 0. \\
\end{array}
\right. 
\end{equation*}
 We extend the recent results from \cite{BL, Bl, Yi} on first-order necessary conditions for optimality problems, moving from finitely many inequality constraints to a countable number of inequality constraints (Corollary \ref{opt}). This extension is achieved using the concept of upper Dini differentiability, including the Gateaux-differentiable case. Note that the recent work in \cite{BB} also deals with optimisation problems with an infinite number of inequality constraints in the framework where the functions are Gateaux-differentiable at the optimal point. Finally, we extend the result of Pourciau established in \cite[Theorem 6, ~p. 445]{Po} from finite number of convex functions in finite dimension to sequences of convex functions in a real vector space (Corollary \ref{pourc0} and Corollary \ref{pourc}).

\vskip5mm

This article is organised as follows. In Section \ref{SS2}, we present some intermediate results that we will use in the rest of the paper. In Section \ref{SS3}, we  give our main results Theorem ~\ref{thm4} and Theorem ~\ref{thm} and their  corollaries. Finally, in Section \ref{SS6}, we give two examples  to illustrate optimisation problems with a finite or infinite number of inequality constraints, in the absence of Gateaux differentiability.

\section{A few intermediate results.} \label{SS2}

  Let $E$ be a real normed vector space and $\Omega$ be a nonempty open subset of $E$. Recall that a function $f: \Omega \to \R$ is said to be upper Dini-differentiable at $x \in \Omega$ if  $D^+ f(x)(u)$ is  finite, for all directions $u \in E$, where 
\[
D^+ f(x)(u) :=\limsup_{t\to 0^+} \frac{f(x+tu)-f(x)}{t}.
\]
We always assume that  $D^+f(x)(0) = 0$. It is easy to see that any real-valued Lipschitz function in some neighborhood of $x\in E$ is  upper Dini-differentiable at $x$ and moreover the map  $u\mapsto D^+ f(\hat{x})(u)$ is Lipschitz on $E$ with the same constant of Lipschitz of $f$ at $x$ (see for instance \cite{BL}). Recall that  a function $p: E \rightarrow \mathbb{R}$ is sublinear if for all $x, y \in E$, $p(x+y) \leq p(x) + p(y)$, and for all $x \in E$, for all $\lambda \in \R^+$, $p(\lambda x) = \lambda p(x)$. The upper Dini-derivative of a function at some point is always positively homogenous but it is not sublinear in general. All functions in our main results will be assumed to have a sublinear upper Dini-differential. It is known from \cite[Lemma 1.2]{Ph}, that for a convex function $f$ on some convex neighbourhood of $x$, the ``right hand " directional derivative of $f$ at $x$ in the direction $u\in E$ defined by 
$$d^+ f(\hat{x})(u):=\lim_{t\to 0^+} \frac{f(\hat{x}+tu)-f(\hat{x})}{t},$$ 
exists for all $u\in E$ and the map $u\mapsto d^+ f(x)(u)$ is sublinear. If moreover $f$ is continuous at $x$, then $u\mapsto d^+ f(x)(u)$ is also continuous (\cite[Corollary 1.7]{Ph}).  Moreover, the following formula hold (see \cite[p. 6]{Ph})
\begin{eqnarray} \label{star0}
d^+f(x)(y-x)\leq f(y)-f(x), \forall y \in E.
\end{eqnarray}
Finally, we see that  if $f$ is convex on some convex neighborhood of $x \in E$, then 
\begin{eqnarray} \label{star}
D^+ f(x)(u) =d^+ f(x)(u), \forall u\in E.
\end{eqnarray}
 Notice also that in the case when $f$ is Gateaux-differentiable at $x$,  then  $D^+ f(x)(u)=d_G f (x)$  the classical Gateaux-differential of $f$ at $x$, which is (by definition) a linear continuous functional on $E$. 

Let $E$ be a real vector normed space, $\Omega \subset E$ an open set, $\hat{x}\in \Omega$, $B(\hat{x},r)$ the open unit ball centered at $\hat{x}$ with radius $r>0$. We denote $\|\cdot\|_{\hat{x},r}$  the seminorm defined on the space of real-valued Lipschitz functions $h$ on $B(\hat{x},r)\cap \Omega$  by
\begin{eqnarray*}\label{eq1}
\| h \|_{\hat{x},r} = \sup_{_{\substack{x, y \in B(\hat{x},r)\cap \Omega \\ x \neq y}}} \frac{|h(x) - h(y)| }{\| x-y \|}.
\end{eqnarray*} 
By convention, $r$ can take the value $+\infty$ (if $h$ is globaly Lipschitz on $\Omega$), in this case we denote $\|h\|_{\textnormal{Lip}}=\| h\|_{\hat{x},+\infty}$ the constant of Lipschitz of $h$ on $\Omega$. 
\vskip5mm
\paragraph{\bf The property $(H)$} For all $n\in \N$, let $h_n:\Omega \to \R$ be a  Lipschitz function on $B(\hat{x},r)\cap \Omega$. We say that a sequence $(h_n)$ satisfies the property $(H)$ in $E$ at $\hat{x}$ if the following assertions hold:

$(i)$ $D^+h_n(\hat{x})$ is sublinear for all $n\in \N$. 

$(ii)$ There exists a function $h_\infty$ (Lipschitz on $B(\hat{x},r)\cap \Omega$), such that $h_{\infty}(\hat{x})=\lim_{n\to+\infty} h_n (\hat{x})$ and $\|h_n-h_\infty\|_{\hat{x},r}\to 0$ when $n\to +\infty$.
\vskip5mm
Notice that from the part $(ii)$ we have, 

$\bullet$ $h_{\infty}(x)=\lim_{n\to+\infty} h_n (x)$, $\sup_{n\in \N}|f_n(x)|<+\infty$ for all $x\in B(\hat{x},r)\cap \Omega$ and $\sup_{n\in \N}\|h_n\|_{\hat{x},+\infty}<+\infty$,

$\bullet$ $\forall \epsilon > 0$, $\exists N^{\epsilon} \in \mathbb{N}$ s.t. $\forall k \geq N^{\epsilon}$, $\forall u\in E\setminus \lbrace 0\rbrace$, $\forall t \in (0, \frac{r}{\|u\|})$,
\begin{eqnarray}\label{eq1}
\frac{h_k (\hat{x}+t u)-h_k (\hat{x})}{t} - \epsilon \| u \| \leq \frac{h_\infty (\hat{x}+t u)-h_\infty (\hat{x})}{t} \leq \frac{h_k (\hat{x}+t u)-h_k (\hat{x})}{t} + \epsilon \| u \|.
\end{eqnarray}

\vskip5mm
We are going to give some propositions and lemmas that will serve to prove Theorem \ref{thm4}.

\begin{proposition} \label{prop1} Let us assume that a sequence $(f_n)$ satisfies the property $(H)$ at $\hat{x}$. Then, for all $u\in E$, we have $$  D^+ f_\infty (\hat{x})(u) = \lim_{k\to +\infty}  D^+ f_k (\hat{x})(u).$$
Consequently, $  D^+ f_\infty (\hat{x})$ is Lipschitz and sublinear. If moreover, we assume that $f_k$ is Gateaux-differentiable at $\hat{x}$ for all $k\in \N$, then $f_\infty$ is also Gateaux-differentiable at $\hat{x}$ and $  D^+ f_\infty (\hat{x})=d_G f_{\infty}(\hat{x})$.
\end{proposition}

\begin{proof} Using the formula in $(\ref{eq1})$ we deduce:
$\forall \epsilon > 0$, $\exists N^{\epsilon} \in \mathbb{N}$ s.t $\forall k \geq N^{\epsilon}$, $\forall u\in E$,
\begin{eqnarray*}
D^+f_k (\hat{x})(u) - \epsilon \| u \| \leq D^+f_\infty (\hat{x})(u) \leq D^+f_k (\hat{x})(u)+ \epsilon \| u \|,
\end{eqnarray*}
which gives the first part of the proposition. Moreover, $D^+ f_\infty (\hat{x})$ is Lipschitz (since $f_\infty $ is Lipschitz) and sublinear as a pointwize limit of a sequence of sublinear functions. If moreover we assume that $f_k$ is Gateaux-differentiable at $\hat{x}$ for all $k\in \N$, then $D^+ f_k (\hat{x})=d_G f_k(\hat{x})$ is linear continuous and so $  D^+ f_\infty (\hat{x})$ is linear continuous as a pointwize limit of a sequence of linear continuous functionals. Hence, $  D^+ f_\infty (\hat{x})=d_G f_{\infty}(\hat{x})$.
\end{proof}

We need the following lemma  which we will use in Proposition \ref{prop2}.

\begin{lemma} \label{lem1} Let $N\geq 0$ be an integer number and $h_0, h_1, ..., h_N : (0,+\infty) \rightarrow \mathbb{R}$ be bounded from below non-decreasing functions. Then we have 
\[
\max_{0\leq i\leq N} \inf_{s > 0} h_i (s) = \inf_{s > 0} \max_{0\leq i\leq N} h_i (s).
\]
\end{lemma}

\begin{proof} By induction, it suffices to give the proof with $N=1$.  It is not difficult to see that 
\begin{eqnarray*}
\max \{ \inf_{s > 0} h_0 (s) , \inf_{s > 0} h_1 (s) \} \leq \inf_{s > 0} \max \{ h_0 (s) , h_1 (s) \}.
\end{eqnarray*}
We are going to prove the inequality: 
\[
\max \{ \inf_{s > 0} h_0 (s) ,\inf_{s > 0} h_1 (s) \} \geq \inf_{s > 0} \max \{ h_0 (s) , h_1 (s) \}.
\]
Indeed, for all $\epsilon > 0$, there are $ s^{\epsilon}_0 > 0$ and $s^{\epsilon}_1 > 0$, such that $\inf_{s > 0} h_0(s) > h_0(s^{\epsilon}_0) - \epsilon / 2 $ and $\inf_{s > 0} h_1(s) > h_1(s^{\epsilon}_1) - \epsilon / 2 $. Then we have
\begin{eqnarray}\label{eq4}
\max \{ \inf_{s > 0} h_0 (s) , \inf_{s > 0} h_1 (s) \} > \max \{ h_0(s^{\epsilon}_0) , h_1(s^{\epsilon}_1) \} - \epsilon
\end{eqnarray}
We take $r_\epsilon = \min \{ s^{\epsilon}_0 , s^{\epsilon}_1 \}>0$. Since $h_0, h_1$ are non-decreasing 
\begin{eqnarray}\label{eq5}
\max \{ h_0(s^{\epsilon}_0) , h_1(s^{\epsilon}_1) \} \geq \max \{ h_0(r_\epsilon) , h_1(r_\epsilon) \} \geq \inf_{s > 0} \max \{ h_0 (s) , h_1 (s) \}
\end{eqnarray}
By combining inequalities $\ref{eq4}$, $\ref{eq5}$ and by passing to the limit when $\epsilon \to 0$, we get that
\[
\max \{ \inf_{s > 0} h_0 (s) , \inf_{s > 0} h_1 (s) \} \geq \inf_{s > 0} \max \{ h_0 (s) , h_1 (s) \}.
\]
\end{proof}

\begin{proposition} \label{prop2} Let us assume that a sequence $(f_n)$ satisfies the property $(H)$ at $\hat{x}$. Then, for all $u\in E$,
$$\limsup_{t\to 0^+} \sup_{k\in \mathbb{N}}\frac{f_k(\hat{x}+t u)-f_k(\hat{x})}{t} =  \sup_{k\in \mathbb{N}}D^+f_k (\hat{x})(u).$$
\end{proposition}

\begin{proof} Let $\epsilon>0$ and choose $N^{\epsilon}$ as in the formula (\ref{eq1}). 
Since 
$$\sup_{k\in \mathbb{N}}\frac{f_k(\hat{x}+t u)-f_k(\hat{x})}{t} = \max \{ \max_{0\leq i\leq  N^{\epsilon}} \frac{f_i(\hat{x}+t u)-f_i(\hat{x})}{t} , \sup_{k \geq N^{\epsilon}} \frac{f_k(\hat{x}+t u)-f_k(\hat{x})}{t} \}$$
Then, we have
\begin{eqnarray*}
\limsup_{t\to 0^+} \sup_{k\in \mathbb{N}}\frac{f_k(\hat{x}+t u)-f_k(\hat{x})}{t} &=&  \inf_{s>0} \sup_{0<t\leq s} \sup_{k\in \mathbb{N}}\frac{f_k(\hat{x}+t u)-f_k(\hat{x})}{t}\\
&=&  \inf_{s>0} \max \{ \max_{0\leq i \leq N^{\epsilon}}\sup_{0<t\leq s} \frac{f_i(\hat{x}+t u)-f_i(\hat{x})}{t}, \\
&& \hspace{15mm} \sup_{0<t\leq s} \sup_{k \geq N^{\epsilon}} \frac{f_k(\hat{x}+t u)-f_k(\hat{x})}{t} \}\\
&:=&  \inf_{s>0} \max_{0\leq i\leq N^{\epsilon}+1} h_i(s),
\end{eqnarray*}
where, for $0\leq i\leq N^{\epsilon}$, $$\forall s\in (0,+\infty), h_i(s) := \sup_{0<t\leq s} \frac{f_i(\hat{x}+t u)-f_i(\hat{x})}{t}\geq -\|f_i\|_{\hat{x},r}\|u\|,$$ and 
$$\forall s\in (0,+\infty), h_{N^{\epsilon}+1}(s) := \sup_{0<t\leq s} \sup_{k \geq N^{\epsilon}} \frac{f_k(\hat{x}+t u)-f_k(\hat{x})}{t}\geq -\inf_{k \geq N^{\epsilon}}\|f_k\|_{\hat{x},r}\|u\|.$$
Since the function $h_i: (0,+\infty)\to \R$ is bounded from below and non-decreasing with respect to $s$ ($0\leq i\leq N^{\epsilon}+1$), by lemma \ref{lem1} and by the definition of limsup,  we get
\begin{eqnarray}\label{eq6}
\limsup_{t\to 0^+} \sup_{k\in \mathbb{N}}\frac{f_k(\hat{x}+t u)-f_k(\hat{x})}{t} &=&  \inf_{s>0} \max_{0\leq i\leq N^{\epsilon}+1} h_i(s) \nonumber\\
&=& \max_{0\leq i \leq N^{\epsilon}+1} \inf_{s>0} h_i(s) \nonumber\\
&=& \max \{\max_{0\leq i\leq N^{\epsilon}} D^+f_i(\hat{x})(u) , \nonumber\\
&& \hspace{5mm} \limsup_{t\to 0^+} \sup_{k \geq N^{\epsilon}} \frac{f_k(\hat{x}+t u)-f_k(\hat{x})}{t} \}
\end{eqnarray}

 Using the formulas in (\ref{eq1}) we see that 
\begin{eqnarray*}
\limsup_{t\to 0^+} \sup_{k \geq N^{\epsilon}} \frac{f_k(\hat{x}+t u)-f_k(\hat{x})}{t} \leq D^+f_\infty (\hat{x})(u) + \epsilon \| u \|.
\end{eqnarray*}
Hence, by the formula (\ref{eq6}), we obtain
\begin{eqnarray*}  \limsup_{t\to 0^+} \sup_{k\in \mathbb{N}}\frac{f_k(\hat{x}+t u)-f_k(\hat{x})}{t} \leq \max \{ \max_{0\leq i\leq N^{\epsilon}} D^+f_i (\hat{x})(u) , D^+f_\infty (\hat{x})(u) + \epsilon \| u \| \}.
\end{eqnarray*}
Using Proposition \ref{prop1}, we have that $D^+f_\infty (\hat{x})(u)\leq \sup_{k\in \N} D^+f_k(\hat{x})(u)$ and so we get 
\begin{eqnarray*}  \limsup_{t\to 0^+} \sup_{k\in \mathbb{N}}\frac{f_k(\hat{x}+t u)-f_k(\hat{x})}{t} &\leq& \max \{ \sup_{_{\substack{k \in \N}}} D^+f_k (\hat{x})(u) , D^+f_\infty (\hat{x})(u) + \epsilon \| u \| \}\\
& \leq & \sup_{_{\substack{k \in \N}}} D^+f_k (\hat{x})(u) + \epsilon \| u \|.
\end{eqnarray*}
By passing to the limit when $\epsilon \to 0$,  we obtain
\begin{eqnarray*}    \limsup_{t\to 0^+} \sup_{k\in \mathbb{N}}\frac{f_k(\hat{x}+t u)-f_k(\hat{x})}{t} \leq \sup_{_{\substack{k \in \N}}} D^+f_k (\hat{x})(u).
\end{eqnarray*}
The inverse inequality is trivial.
\end{proof}

Recall that $c(\N)$ denotes the space of all converging sequences of real numbers equipped with the sup-norm.  Recall also that the topological dual of $c(\N)$ is isometrically isomorphic to the space $\ell^1(\N)$ (see \cite{DS}). The isomorphism of $\ell ^{1}(\N)$ with $c^*(\N)$ is given as follows. If $(\alpha_{\infty},\alpha_0,\alpha_1,...)\in \ell ^{1}(\N)$ then the pairing with an element $(y_0,y_1,...)\in c(\N)$ is given by 
$$ \displaystyle \alpha_{\infty}\lim _{n\to \infty }y_{n}+\sum _{i=0}^{\infty }\alpha_{i}y_{i}.$$

We denote $ \ell^1_+(\N)$ and $c_+(\N)$ the closed convex positive cone of $\ell^1(\N)$ and $c(\N)$ respectively. If $A$ is a nonempty subset of $c(\N)$, we denote by $\textnormal{cone}(A)$ the convex conique hull of $A$ and by $\overline{A}$ we denote the closure of $A$.  We need the following lemma  in the proof of Theorem \ref{thm4}.

\begin{lemma} \label{lem2} Let $X$ be a real vector space and $C$ be a nonempty convex subset of $X$. Let $h_n :C\to \R$, $n\in \N$, be a sequence of convex functions pointwize converging to some (convex) function $h_{\infty}$. Then, there exists a sequence $(\alpha_{\infty},\alpha_0,\alpha_1,...) \in \ell^1_+(\N)$ such that $\alpha_{\infty}+\sum_{n\geq 0} \alpha_n=1$ and 
$$\inf_{u\in C} \left(\alpha_{\infty} h_{\infty}(u)+ \sum_{n\geq 0} \alpha_{n} h_n(u) \right)  = \inf_{u\in C} \sup_{n\in \N} h_n(u).$$ 
\end{lemma}

\begin{proof} Set $r:= \inf_{u\in C} \sup_{n\in \N} h_n(u)$. Since, $h_n$ pointwize converges, we see that $r<+\infty$. We always  have 
$$\inf_{u\in C} \left(\alpha_{\infty} h_{\infty}(u)+ \sum_{n\geq 0} \alpha_{n} h_n(u) \right) \leq r.$$ 
We prove the inverse inequality. If $r=-\infty$, then the formulas is trivially true with any sequence  $(\alpha_{\infty},\alpha_0,\alpha_1,...) \in \ell^1_+(\N)$ such that $\alpha_{\infty}+\sum_{n\geq 0} \alpha_n=1$. Suppose that $r>-\infty$. Set $A:=\lbrace \left (h_n(u) - r\right)_{n\in \N}\in \R^{\N}: u\in C\rbrace\subset c(\N)$ (since $(h_n)$ pointwize converges to $h_{\infty}$). Let us show that the sequence $-1_\N:=(-1,-1,...,-1,...) \not \in \overline{\textnormal{cone}(A)+c_+(\N)}$. Suppose that the contrary holds. Then, there exists $(a_n)\in \textnormal{cone}(A)$ and $(x_n)\in c_+(\N)$ such that
$$\|(a_n)+(x_n)+1_\N \|_{\infty}<\frac{1}{2}.$$
By the definition of $\textnormal{cone}(A)$, there exists some $k\geq 1$, $\lambda_1,...,\lambda_k \geq 0$ and $u_1,...,u_k\in C$ such that $a_n= \sum_{i=1}^k \lambda_i \left (h_n(u_i)-r\right)$ for all $n\in \N$. From the above inequality we get that 
$$ \sum_{i=1}^k \lambda_i \left (h_n(u_i)-r\right) +x_n +1<\frac{1}{2}, \hspace{1mm} \forall n\in \N.$$
Since $x_n\geq 0$, then there exists some $i\in \lbrace 1,...,k\rbrace$ such that $\lambda_i\neq 0$. Set $\gamma_i:=\frac{\lambda_i}{\sum_{j=1}^k \lambda_j}$. Using the convexity of $h_n$, we get
$$h_n(\sum_{i=1}^k \gamma_i u_i)< \frac{1}{\sum_{j=1}^k \lambda_j}(-\frac{1}{2}-x_n) +r<-\frac{1}{2\sum_{j=1}^k \lambda_j} +r, \hspace{1mm} \forall n\in \N.$$
We deduce, using the definition of $r$ that $r\leq -\frac{1}{2\sum_{j=1}^k \lambda_j} +r$ which is a contradiction. Hence, $-1_\N \not \in \overline{\textnormal{cone}(A)+c_+(\N)}$ and so by the Hanh-Banach separation theorem and the fact that $\overline{\textnormal{cone}(A)+c_+(\N)}$ is a cone, there exists $(\alpha_{\infty},\alpha_0, \alpha_1,...)\in l^1(\N)\setminus \lbrace 0\rbrace$ such that 
\begin{eqnarray} \label{INE}
\alpha_{\infty}\lim_{n\to +\infty} y_n + \sum_{n\geq 0} \alpha_{n} y_n\geq 0, \hspace{1mm} \forall (y_n) \in \overline{\textnormal{cone}(A)+c_+(\N)}.
\end{eqnarray}

\noindent {\bf Claim.} We have  $\alpha_{\infty}\geq 0$ and $\alpha_n\geq 0$ for all $n\geq 0$. 

\begin{proof}[Proof of the claim] Since $e_k:=(\delta_n^k)\in c_+(\N)\subset \overline{\textnormal{cone}(A)+c_+(\N)}$ (where $\delta_n^k$ is the Kronecker symbol satisfying $\delta_n^k=1$ if $n=k$ and $0$ if $n\neq k$), we see from the above inequality that $\alpha_n\geq 0$ for all $n\geq 1$. On the other hand, for all $K\in \N$, define $y_n=0$ if $n\leq K$ and $y_n=1$ if $n\geq K$. Then, $(y_n)\in c_+(\N)\subset \overline{\textnormal{cone}(A)+c_+(\N)}$ and by the above inequality we have that $\alpha_{\infty} + \sum_{n\geq K} \alpha_{n}\geq 0$ for all $K\in \N$, which implies that $\alpha_{\infty}\geq 0$ by taking the limit when $K\to +\infty$, since the series $\sum_{n\geq 0} \alpha_{n}$ is a convergent. This completes the proof of the claim.
\end{proof}

Now, using $(\ref{INE})$ by taking sequences $(y_n)\in A\subset \overline{\textnormal{cone}(A)+c_+(\N)}$, we obtain 
$$\alpha_{\infty}(h_{\infty}(u)-r)+\sum_{n\geq 0} \alpha_{n} \left (h_n(u)-r\right)\geq 0, \hspace{1mm} \forall u\in C.$$
Dividing by $\alpha_{\infty}+\sum_{n\geq 0} \alpha_{n} >0$ if necessary, we can assume without loss of generality that $\alpha_{\infty}+\sum_{n\geq 0} \alpha_{n}=1$ and we get using the above inequality that, 
$$\inf_{u\in C} \left(\alpha_{\infty} h_{\infty}(u)+\sum_{n\geq 0} \alpha_{n} h_n(u) \right) \geq r:= \inf_{u\in C} \sup_{n\in \N} h_n(u).$$ 
\end{proof}

\section{The main results} \label{SS3}
This section is devoted to the main results of the article. Several new types of alternative theorems will be given and applications to convex and non-convex optimization problems will be derived.  We begin with Theorem \ref{thm4}. We will then give some applications to the alternative theorem for Lipschitz invex functions and  to the problem of optimisation with a countable number of inequality constraints (Corollary \ref{cor3}, Corollary \ref{opt} and Corollary \ref{pourc}).

\subsection{Alternative theorem for sequences of functions}
We give the proof of our main result. 

\begin{theorem} \label{thm4}  Let us assume that a sequence $(h_n)$ satisfies the property $(H)$ at $\hat{x}$. Then, either 
$$\{ (x,t)\in \Omega\times [0,1]: \sup_{n\in \N} \left(h_n(x) - (1-t)h_n(\hat{x})\right) <0 \}\neq \emptyset,$$
or there exits a sequence $(\alpha_{\infty},\alpha_0,\alpha_1,...) $ such that $\alpha_{\infty}\geq 0$ and $\alpha_n\geq 0$ for all $n\in \N$ and satisfying:  

$(i)$ $\alpha_{\infty}+\sum_{n\geq 0} \alpha_n=1$,

$(ii)$ $\alpha_{\infty} h_{\infty}(\hat{x})+\sum_{n\geq 0} \alpha_{n}h_n(\hat{x})\geq 0$,

$(ii)$ $ \alpha_{\infty} D^+h_{\infty}(\hat{x})(u)+\sum_{n\geq 0} \alpha_{n} D^+h_n(\hat{x})(u) \geq 0, \forall u\in E.$

$\bullet$ If moreover, we assume that there exists $w\in E$ such that $ \sup_{n\geq 1} D^+ h_n(\hat{x})(w) < 0$, then  $ \alpha_{0}\neq 0$. 

$\bullet \bullet$  If we assume that $h_n$ is Gateaux-differentiable at $\hat{x}$ for all $n\in \N$, then $ \alpha_{\infty} d_G h_{\infty}(\hat{x})(u) + \sum_{n\geq 0} \alpha_{n} d_G h_n(\hat{x})(u)= 0$, for all $x\in E$.

\end{theorem}

\begin{proof}  Suppose that $\lbrace (x,t)\in \Omega \times [0,1]: \sup_{n\in \N} \left(h_n(x) - (1-t)h_n(\hat{x})\right) <0 \rbrace=\emptyset$ then we have $\sup_{n\in \N}\lbrace h_n (x) -(1-t)h_n(\hat{x}) \rbrace\geq 0$ for all $x\in \Omega$ and for all $t\in [0,1]$. Thus, for all  $u \in E$ there exists $\delta_u \in(0,1]$ such that for all $t \in (0,\delta_u]$, we have $\hat{x}\pm tu\in \Omega$ and 
\[
 \sup_{n \in \N } \left( h_n (\hat{x} \pm tu) - (1-t)h_n (\hat{x})\right) \geq 0.
\]
It follows in particular that,
\[
\limsup_{t\to 0^+}  \sup_{n \in \N }  \frac{h_n (\hat{x} + tu) - (1-t)h_n (\hat{x})}{t} \geq 0, \forall u\in E.
\]
For each $u\in E$ and $n\in \N$, define $\phi^u_n(t):=h_n(\hat{x}+tu)+th_n(\hat{x})$ for all $t\in (-\delta_u,\delta_u)$. We see that the sequence $(\phi^u_n)$ satisfies the property $(H)$ in $\R$ at the point $0$, since $(h_n)$ satisfies $(H)$ in $E$ at $\hat{x}$. By applying Proposition \ref{prop2} to the sequence $(\phi^u_n)$, we obtain,
$$\sup_{n \in \N} D^+\phi^u_n(0)=\limsup_{t\to 0^+} \sup_{n \in \N} \frac{\phi^u_n(t) -\phi^u_n(0)}{t}.$$
In other words, for each $u\in E$,
$$\sup_{n \in \N} \left (D^+h_n (\hat{x})(u)+h_n(\hat{x})\right )=\limsup_{t\to 0^+} \sup_{n \in \N} \frac{h_n(\hat{x}+t u)-(1-t)h_n(\hat{x})}{t}.$$
Therefore, we have 
\begin{eqnarray}\label{Peq3} 
\inf_{u\in E} \sup_{n \in \N} \left (D^+h_n (\hat{x})(u)+h_n(\hat{x})\right )\geq  0. 
\end{eqnarray}
Define $g_n:= D^+h_n(\hat{x})(\cdot)+h_n(\hat{x})$ for all $n\in \N$. Then, $g_n$ is convex for all $n\in \N$ and pointwize converges to $g_{\infty}:=D^+h_{\infty}(\hat{x})(\cdot)+h_{\infty}(\hat{x})$ by Proposition \ref{prop1}.  We apply Lemma ~\ref{lem2} with the sequence $(g_n)$, there exists a sequence $(\alpha_{\infty},\alpha_0,\alpha_1,...)\in \ell^1_+(\N)$ such that $\alpha_{\infty}+\sum_{n\geq 0} \alpha_n =1$ and
\begin{eqnarray*}
\inf_{u\in E} \left (\alpha_{\infty} g_{\infty}(u)+\sum_{n\geq 0} \alpha_{n} g_n(u) \right)&=& \inf_{u\in E} \sup_{n\in \N} g_n(u)\geq 0 \hspace{1mm} (\textnormal{by using } (\ref{Peq3})).
\end{eqnarray*}
Thus, we have for all $u\in E$,
$$ \left(\alpha_{\infty} D^+h_{\infty}(\hat{x})(u)+\sum_{n\geq 0} \alpha_{n} D^+h_n(\hat{x})(u) \right) +\left(\alpha_{\infty} h_{\infty}(\hat{x})+\sum_{n\geq 0} \alpha_{n}h_n(\hat{x})\right)\geq 0.$$
From this inequality we can derive the following two pieces of information: 

-- Firstly, by taking $u=0$, we get 
$$\alpha_{\infty} h_{\infty}(\hat{x})+\sum_{n\geq 0} \alpha_{n}h_n(\hat{x})\geq 0.$$

-- Secondely,  since the map $\alpha_0 D^+h_{\infty}(\hat{x})(\cdot)+\sum_{n\geq 0} \alpha_{n} D^+h_n(\hat{x})(\cdot)$ is positively homogeneous, we see that for all $u\in E$ we have
$$ \alpha_{\infty} D^+h_{\infty}(\hat{x})(u)+\sum_{n\geq 0} \alpha_{n} D^+h_n(\hat{x})(u)\geq 0.$$
This ends the first part of the proof.

$\bullet$ If moreover, we assume that there exists $w\in E$ such that $\sup_{n\geq 1} D^+ h_n(\hat{x})(w) < 0$, then $\alpha_{0}\neq 0$. Indeed, suppose by contradiction that $\alpha_{0}= 0$, then there exists some integer number $N\in \N \setminus \{0\}$ such that $\max(\alpha_{\infty}, \alpha_N)>0$ (since, $\alpha_{\infty}+\sum_{n\geq 0} \alpha_n =1$). Since $\alpha_{\infty}\geq 0$, $\alpha_n\geq 0$ for all $n\in \N$ and $D^+ h_{\infty}(\hat{x})(w)\leq \sup_{n\geq 1} D^+ h_n(\hat{x})(w) < 0$, it follows that $$0>  \alpha_{\infty} D^+ h_{\infty} (\hat{x})(w)+ \alpha_N D^+ h_N (\hat{x})(w)+\sum_{n\in \N\setminus \{0, N\}} \alpha_{n} D^+ h_n(\hat{x})(w) \geq 0,$$ which is a contradiction.

$\bullet \bullet$  If we assume that $h_n$ is Gateaux-differentiable at $\hat{x}$ for all $n\in \N$, then $\alpha_{\infty} d_G h_{\infty}(\hat{x})(u)+\sum_{n\geq 0} \alpha_{n} d_G h_n(\hat{x})(u)=0$. Indeed, we use the fact that $D^+ h_n(\hat{x})=d_G h_n (\hat{x})$ for all $n\in \N$ and $D^+ h_{\infty}(\hat{x})=d_G h_{\infty} (\hat{x})$ (see Proposition \ref{prop1}), we get that $ \alpha_{\infty} d_G h_{\infty}(\hat{x})(u) + \sum_{n\geq 0} \alpha_{n} d_G h_n(\hat{x})(u)\geq 0$ for all $u\in E$. Since the Gateaux-differentials are linear, we obtain $ \alpha_{\infty} d_G h_{\infty}(\hat{x})(u) + \sum_{n\geq 0} \alpha_{n} d_G h_n(\hat{x})(u)= 0$.
\end{proof}

The following result extend the alternative theorem given by Fan, Glicksberg, and Hoffman \cite{FGH} to a sequence $(h_n)$ of convex functions (without continuity) pointwize converging to some function and defined on a nonempty convex subset of some real vector space (not necessarily a normed space).

\begin{theorem} \label{thm}  Let $X$ be a real vector space and $C$ be a nonempty convex subset of $X$. Let $h_n :C\to \R$, $n\in \N$, be a sequence of convex functions pointwize converging to some function $h_{\infty}$. Then,  one and only one of the following alternatives is true: 

$A1)$ there exists $ u_0 \in C $ such that $\sup_{n\in \N} h_n(u_0) <0, $

$A2)$ there exits a sequence of real numbers $(\alpha_{\infty}, \alpha_{0},\alpha_{1},...) $  such that $\alpha_{\infty}\geq 0$, $\alpha_{n}\geq 0$ for all $n\in \N$,  $\alpha_{\infty}+\sum_{n\geq 0} \alpha_{n}=1$ and
 $\alpha_{\infty} h_{\infty}(u)+\sum_{n\geq 0} \alpha_{n} h_n (u)\geq 0,$ for all $u\in E$.

\end{theorem}
\begin{proof}  From Lemma \ref{lem2}, there exists a sequence $(\alpha_{\infty},\alpha_0,\alpha_1,...) \in \ell^1_+(\N)$ such that $\alpha_{\infty}+\sum_{n\geq 0} \alpha_n=1$ and 
$$\inf_{u\in C} \left(\alpha_{\infty} h_{\infty}(u)+ \sum_{n\geq 0} \alpha_{n} h_n(u) \right)  = \inf_{u\in C} \sup_{n\in \N} h_n(u).$$ 
Then, either $\inf_{u\in C} \sup_{n\in \N} h_n(u)<0$ or $\inf_{u\in C} \sup_{n\in \N} h_n(u)\geq 0$, which gives the alternatives $A1)$ and $A2)$ respectively.
\end{proof}

\begin{remark} Notice that the case of a finite number of functions $h_0,h_1, ..., h_N$ can be deduced from our results by considering the sequence $(\widetilde{h}_n)$ with $\widetilde{h}_n=h_n$ for all $0\leq n\leq N$ and $\widetilde{h}_n=h_N$ for all $n\geq N$.
\end{remark}
\subsection{Alternative theorem of invex functions} \label{invex}
Real-valued invex functions were introduced by Hanson \cite{Ha} as a generalization of convex functions: A function $f : E \to \R$, Gateaux differentiable at $\hat{x}\in E$, is said to be invex at $\hat{x}$ if there exists a vector-valued function $\mu_{\hat{x}}: E \to E$ such that
$$ d_G f(\hat{x})(\mu_{\hat{x}}(x)) \leq f(x)-f(\hat{x}), \forall x\in E.$$
We say that $f$ is invex if it is invex at each point. We refer to \cite{YS, Re, BL} for other type of invexity and the relationship between them. In the following definition, we extend the invexity notion to a sequence of Lipschitz functions using the upper Dini derivative. 

\begin{definition} Let $E$ be a normed real vector space, $\Omega \subset E$ an open set and $\hat{x}\in \Omega$.  Let $f_n : \Omega \to \R$ be a  function which is Lipschitz on $B(\hat{x},r)\cap \Omega$ (for some $r>0$), for every $n\in \N$. We say that the sequence $(f_n)$ is invex at $\hat{x}$ if there exists a vector-valued function $\mu_{\hat{x}} : E \to E$ such that
$$ D^+ f_n(\hat{x})(\mu_{\hat{x}}(x)) \leq f_n(x)-f_n(\hat{x}), \forall x\in \Omega, \forall n\in \N.$$
\end{definition}

\begin{example} \label{ho} Let $E$ be a normed real vector space.

$1)$ Any sequence $(f_n)$ of convex Lipschitz functions on $E$ is invex at any point $\hat{x}\in E$, with $\mu_{\hat{x}}(x)=x-\hat{x}$ for every $x\in E$. This follows from the formulas in $(\ref{star0})$ and $(\ref{star})$.


$2)$ Let $(p_n)$ be a sequence of linear continuous functionals on $E$. For all $n\in \N$, let $f_n : E \to \R$ defined by $f_n(x)=\textnormal{actan}^2(p_n(x))+\|x\|$ for all $x\in \R$. Then, for each $n\in \N$, $f_n$ is a Lipschitz function, but neither convex nor  Gateaux-differentiable at $0$.  However, it is easy to see  that $D^+ f_n (0)(x)=d^+ f_n (0)(x)=\|x\|\leq f_n(x)=f_n(x)-f_n(0)$ for all $x\in E$ and all $n\in \N$. Thus, the sequence $(f_n)$ is invex at $0$, with $\mu_{0}(x)=x$ for every $x\in E$. 
\end{example}
We obtain the following alternative theorem for a sequence of Lipschitz invex functions.

\begin{corollary} \label{cor3}  Let us assume that a sequence $(h_n)$ satisfies the property $(H)$ at $\hat{x}$. Suppose moreover that $(h_n)$ is invex at $\hat{x}$. Then, one of the following alternatives is true:

$A1)$ $\{ (x,t)\in \Omega\times [0,1]: \sup_{n\in \N} \left(h_n(x) - (1-t)h_n(\hat{x})\right) <0 \}\neq \emptyset$,

$A2)$ there exits a sequence $(\alpha_{\infty},\alpha_0,\alpha_1,...) $ such that:

$(i)$ $\alpha_{\infty}\geq 0$, $\alpha_n\geq 0$ for all $n\in \N$ and $\alpha_{\infty}+\sum_{n\geq 0} \alpha_n=1$,

$(ii)$ $\alpha_{\infty} h_{\infty}(\hat{x}) + \sum_{n\geq 0} \alpha_{n} h_n(\hat{x}) \geq 0$,

$(iii)$ the function $ \alpha_{\infty} h_{\infty}+\sum_{n\geq 0} \alpha_{n} h_n$ has a global minimum on $\Omega$ at $\hat{x}$.

\end{corollary}

\begin{proof} Clearly, alternatives $A1)$ and $A2)$ cannot be true at the same time. Suppose that $A1)$ is not true, then from Theorem \ref{thm4} there exits a sequence $(\alpha_{\infty},\alpha_0,\alpha_1,...) $ such that :  

$(a)$ $\alpha_{\infty}\geq 0$, $\alpha_n\geq 0$ for all $n\in \N$ and $\alpha_{\infty}+\sum_{n\geq 0} \alpha_n=1$, 

$(b)$ $\alpha_{\infty} h_{\infty}(\hat{x})+\sum_{n\geq 0} \alpha_{n}h_n(\hat{x})\geq 0$,

$(c)$ $ \alpha_{\infty} D^+h_{\infty}(\hat{x})(u)+\sum_{n\geq 0} \alpha_{n} D^+h_n(\hat{x})(u) \geq 0, \forall u\in E.$ 

Using the fact that $(h_n)$ is invex at $\hat{x}\in E$ together with Proposition \ref{prop1}, we see that 
$$ D^+ h_n(\hat{x})(\mu_{\hat{x}}(x)) \leq h_n(x)-h_n(\hat{x}), \forall x\in \Omega, \forall n\in \N,$$
and
$$ D^+ h_{\infty}(\hat{x})(\mu_{\hat{x}}(x)) \leq h_{\infty}(x)-f_{\infty}(\hat{x}), \forall x\in \Omega.$$
Thus, from $(c)$ we get that for all $x\in \Omega$,
\begin{eqnarray*}
0 &\leq& \alpha_{\infty} D^+ h_{\infty}(\hat{x})(\mu_{\hat{x}}(x)) +\sum_{n\geq 0} \alpha_n D^+ h_n(\hat{x})(\mu_{\hat{x}}(x)) \\
&\leq& \alpha_{\infty} \left (h_{\infty}(x) - h_{\infty}(\hat{x}) \right)+ \sum_{n\geq 0} \alpha_{n} \left (h_n(x) - h_n(\hat{x}) \right).
\end{eqnarray*}
In other words, the function $ \alpha_{\infty} h_{\infty}+\sum_{n\geq 0} \alpha_{n} h_n$ has a global minimum on $\Omega$ at $\hat{x}$.
\end{proof}

\subsection{Application to optimization problem with countable number of inequality constraints} \label{SS4}

This section is a continuation of recent works on optimisation problems with inequality constraints obtained in \cite{BL, Bl, BB}. We consider the first-order necessary conditions for  optimality problems with inequality constraints in a normed vector space in the form of John's theorem and in the form of Karush-Kuhn-Tucker's theorem. Let us start by considering the following finite-dimensional problem with a finite number of constraints:
\begin{equation*}
(\mathcal{P}_N)
\left \{
\begin{array}
[c]{l}
\min f_0\\
x\in \Omega\\
\forall 1\leq n\leq N, f_{n}(x) \leq 0. 
\end{array}
\right. 
\end{equation*}
where $\Omega$ is a nonempty open subset of $\R^p$. 

It is well-knowm \cite{Po}  (see also \cite[Chapter 3, Scetion 3.2]{ATF} and \cite{Hh}) that the classical necessary conditions of optimality says that if  $\hat{x}\in \Omega$ is an optimal solution of $(\mathcal{P}_N)$ and that $f_0,f_1,...,f_N$ are Fr\'echet-differentiable at $\hat{x}$ ($d_F f_i(\hat{x})$ denotes the Fr\'echet-differential of $f_i$ at $\hat{x}$), then there exists  $\lambda_0\geq 0, \lambda_1\geq 0,...\lambda_N\geq 0$ such that

$(i)$ $(\lambda_0,\lambda_1,...,\lambda_N)\neq (0,0,...,0)$ (a non-trivial Lagrange multipliers),

$(ii)$ $\lambda_i f_i(\hat{x})=0$ for all $1\leq i\leq N$,

$(iii)$ $\sum_{i=0}^N \lambda_i d_F f_i(\hat{x})=0$ (``condition of Fritz John''). 

If moreover, we assume that there exists $w\in \R^p$ such that $d_F f_i(\hat{x})(w)<0$ for all $1\leq i\leq N$, then we can assume that $\lambda_0\neq 0$ (``condition of Karush-Kuhn-Tucker'').

\vskip5mm
 In \cite{Bl}, J. Blot extended the above conditions from Fr\'echet differentiability to Gateaux differentiablitity (see also \cite{Yi} in infinite dimentional). There is another way to generalize the assumption of Fr\'echet diﬀerentiability by using locally Lipschitz mappings and the Clarke's subdifferential or approximate subdifferential (see \cite{Cf, JT1,JT2, Jourani}). However, it is known that in general the Clarke's subdifferential of a Gateaux differentiable function does not always coincide with the Gateaux derivative of the function. Recall that unlike the Clarke's subdifferential, the upper Dini differentiability coincides with the Gateaux differentiability for Gateaux-differentiable functions. Recently the authors extended in \cite{BL} the result of Blot \cite{Bl} from Gateaux differentiability to the more general  upper Dini differentiability. 

Following on from this work and using our result in Theorem \ref{thm4}, the aim is to find necessary optimality conditions involving Lagrange multipliers for optimisation problems with inequality constraints in a normed vector space. We consider a countable number of inequality constraints for Lipschitz functions around the optimal solution that are not necessarily Gateaux differentiable, thus extending the recent results in \cite{BL, Bl}.  

\paragraph{\bf The non-convex framework in a normed real vector space} Let $E$ be a real normed vector space, $\Omega$ be a nonempty open subset of $E$, let $(f_n)$, be a sequence of real-valued functions on $\Omega$. Consider the following problem:
\begin{equation*}
(\mathcal{P}_\infty)
\left \{
\begin{array}
[c]{l}
\min f_0\\
x\in \Omega\\
\forall n\geq 1, f_{n}(x) \leq 0.
\end{array}
\right. 
\end{equation*}

\begin{corollary} \label{opt}   Let us assume that the assumption $(H)$ is satisfied for the sequence $(f_n)$ at $\hat{x}$.  Suppose that $\hat{x}$ is an optimal solution of the problem $(\mathcal{P}_\infty)$. Then, there exits a sequence $(\alpha_{\infty},\alpha_0,\alpha_1,...) $ such that:

$(i)$ $\alpha_{\infty}\geq 0$, $\alpha_n\geq 0$ for all $n\in \N$ and $\alpha_{\infty}+\sum_{n\geq 0} \alpha_n=1$,

$(ii)$ $\alpha_{\infty} f_{\infty}(\hat{x})=0$ and $ \alpha_{n}f_n(\hat{x})= 0$ for all $n\geq 1$,

$(iii)$ $ \alpha_{\infty} D^+f_{\infty}(\hat{x})(u)+\sum_{n\geq 0} \alpha_{n} D^+f_n(\hat{x})(u) \geq 0, \forall u\in E.$

$\bullet$ If moreover, we assume that there exists $w\in E$ such that $ \sup_{n\geq 1} D^+ f_n(\hat{x})(w) < 0$, then $ \alpha_{0}\neq 0$. 

$\bullet \bullet$  If we assume that $h_n$ is Gateaux-differentiable at $\hat{x}$ for all $n\in \N$, then $ \alpha_{\infty} d_G f_{\infty}(\hat{x})(u) + \sum_{n\geq 0} \alpha_{n} d_G f_n(\hat{x})(u)= 0$.

\end{corollary}

\begin{proof} Define the sequence $(h_n)$ as follows: $h_0:=f_0-f_0(\hat{x})$ and $h_n:=f_n$ for all $n\geq 1$. Clearly, the sequence $(h_n)$ also satisfies the property $(H)$ and $\hat{x}$ is also an optimal solution of the problem $(\widetilde{\mathcal{P}}_\infty)$:

\begin{equation*}
(\widetilde{\mathcal{P}}_\infty)
\left \{
\begin{array}
[c]{l}
\min h_0\\
x\in \Omega\\
\forall n\geq 1, h_{n}(x) \leq 0.
\end{array}
\right. 
\end{equation*}

It follows that $h_n(\hat{x})\leq 0$ for all $n\in \N$ ($h_0(\hat{x})=0$) and since $\hat{x}$ is a solution of $(\widetilde{\mathcal{P}}_\infty)$, we see that $$\lbrace (x,t)\in \Omega \times [0,1]: \sup_{n\in \N} \left(h_n(x) - (1-t)h_n(\hat{x})\right) <0\rbrace=\emptyset.$$
We conclude using Theorem \ref{thm4} applied with the sequence $(h_n)$, observing that the part $(ii)$ of Theorem \ref{thm4} together with the fact that $ h_n(\hat{x})\leq 0$ for all $n\geq 0$, implies that $\alpha_{\infty} h_{\infty}(\hat{x})=0$ and $ \alpha_{n}h_n(\hat{x})= 0$ for all $n\geq 0$. Equivalently, (since $h_0(\hat{x})=0$) $\alpha_{\infty} f_{\infty}(\hat{x})=0$ and $ \alpha_{n} f_n(\hat{x})= 0$ for all $n\geq 1$.

\end{proof}

\paragraph{\bf The convex framework in a real vector space} We give the following extension of Pourciau result in \cite[Theorem 6, p. 445]{Po}. Here, the finite number of convex functions in finite dimension is extended to a countable number of convex functions in real  vector space.

\begin{corollary}  \label{pourc0} Let $X$ be a real vector space and $C$ be a nonempty convex subset of $X$. Let $g_n :C\to \R$, $n\in \N$, be a sequence of convex functions pointwize converging to some function $g_{\infty}$. Suppose that $\hat{x}\in C$ is an optimal solution of the following problem

\begin{equation*}
(\mathcal{P}_\infty)
\left \{
\begin{array}
[c]{l}
\min g_0\\
x\in C\\
\forall n\geq 1, g_{n}(x) \leq 0 \\
\end{array}
\right. 
\end{equation*}

Then,  there exits a sequence $(\alpha_{\infty},\alpha_0,\alpha_1,...) $ such that:

$(i)$ $\alpha_{\infty}\geq 0$, $\alpha_n\geq 0$ for all $n\in \N$ and $\alpha_{\infty}+\sum_{n\geq 0} \alpha_n=1$,

$(ii)$ $\alpha_{\infty} g_{\infty}(\hat{x})=0$ and $ \alpha_{n}g_n(\hat{x})= 0$ for all $n\geq 1$,

$(iii)$ $ \alpha_{\infty} g_{\infty}+\sum_{n\geq 0} \alpha_{n} g_n$, has a global minimum on $C$ at $\hat{x}$.

$\bullet$ If moreover, we assume that there exists $w\in C$ such that $ \sup_{n\geq 1} (f_n(w) -f_n(\hat{x}))< 0$, then $ \alpha_{0}\neq 0$. In this case, the conditions $(i)$, $(ii)$ and $(iii)$ are  necessary and sufficient for $\hat{x}$ to be a solution of $(\mathcal{P}_\infty)$.

\end{corollary}
\begin{proof} Define the sequence $(f_n)$ as follows: $f_0:=g_0-g_0(\hat{x})$ and $f_n:=g_n$ for all $n\geq 1$. Clearly,  $\hat{x}$ is also an optimal solution of the problem $(\widetilde{\mathcal{P}}_\infty)$:

\begin{equation*}
(\widetilde{\mathcal{P}}_\infty)
\left \{
\begin{array}
[c]{l}
\min f_0\\
x\in C\\
\forall n\geq 1, f_{n}(x) \leq 0.
\end{array}
\right. 
\end{equation*}

It follows that $f_n(\hat{x})\leq 0$ for all $n\in \N$ ($f_0(\hat{x})=0$) and since $\hat{x}$ is a solution of $(\widetilde{\mathcal{P}}_\infty)$, we see that $$\lbrace (x,t)\in C \times [0,1]: \sup_{n\in \N} \left(f_n(x) - (1-t)f_n(\hat{x})\right) <0\rbrace=\emptyset.$$
Let us fixe $t_0\in (0,1)$, then 
$$\inf_{u\in E} \sup_{n\in \N} \left(f_n(u) - (1-t_0)f_n(\hat{x})\right) \geq 0.$$
By applying Lemma \ref{lem2} to the sequence $h_n:=f_n- (1-t_0)f_n(\hat{x})$, there exists a sequence $(\alpha_{\infty},\alpha_0,\alpha_1,...) \in \ell^1_+(\N)$ such that $\alpha_{\infty}+\sum_{n\geq 0} \alpha_n=1$ and 
$$\inf_{u\in C} \left(\alpha_{\infty} h_{\infty}(u)+ \sum_{n\geq 0} \alpha_{n} h_n(u) \right)  = \inf_{u\in C} \sup_{n\in \N} h_n(u).$$ 
In other words, 
\begin{eqnarray*}
\alpha_{\infty} (f_{\infty}(u)-f_{\infty}(\hat{x}))+ \sum_{n\geq 0} \alpha_{n} (f_n(u)-f_n(\hat{x})) +t_0\left(\alpha_{\infty} f_{\infty}(\hat{x})+ \sum_{n\geq 0} \alpha_{n} f_n(\hat{x}) \right)\geq 0.
\end{eqnarray*}
By applying the above inequality with $u=\hat{x}$, we get that $$\alpha_{\infty} f_{\infty}(\hat{x})+ \sum_{n\geq 0} \alpha_{n} f_n(\hat{x}) \geq 0.$$ Since $f_n(\hat{x})\leq 0$ and $\alpha_n\geq 0$ for all $n\in \N$, we deduce that $\alpha_{\infty} f_{\infty}(\hat{x})=0$ and $\alpha_{n} f_n(\hat{x})= 0$ for all $n\geq 0$. Then, we deduce that $\alpha_{\infty} f_{\infty}+ \sum_{n\geq 0} \alpha_{n} f_n$ has a global minimum on $C$ at $\hat{x}$. Since,  $f_0:=g_0-g_0(\hat{x})$ and $f_n:=g_n$ for all $n\geq 1$, we give the first part of the corollary.  Suppose now that there exists $w\in E$ such that $ \sup_{n\geq 1} (f_n(w) -f_n(\hat{x}))< 0$. Then, using $(iii)$ we easily see that $\alpha_0\neq 0$. In this case, we see easily that $(i)$, $(ii)$ and $(iii)$ implies that $\hat{x}$ is a solution of $(\mathcal{P}_\infty)$.
\end{proof}

For Lipschitz convex functions in a normed vector space, we obtain the following corollary.

\begin{corollary}  \label{pourc} Let us assume that $(f_n)$ is a sequence of real-valued Lipschitz convex functions on a convex open subset $\Omega$ of $E$ satisfying the property $(H)$ at $\hat{x}\in \Omega$. Assume that there exists $w\in E$ such that $$ \sup_{n\geq 1} d^+ f_n(\hat{x})(w) < 0.$$ Then, the following assertions are equivalente. 

$(a)$ $\hat{x}$ is a solution of the problem:

\begin{equation*}
(\mathcal{P}_\infty)
\left \{
\begin{array}
[c]{l}
\min f_0\\
x\in \Omega\\
\forall n\geq 1, f_{n}(x) \leq 0.
\end{array}
\right. 
\end{equation*}

$(b)$  there exits a sequence $(\beta_{\infty},\beta_0,\beta_1,...) $ such that:

$(i)$ $\beta_{\infty}\geq 0$, $\beta_n\geq 0$ for all $n\in \N$, 

$(ii)$ $\beta_{\infty} f_{\infty}(\hat{x})=0$ and $\beta_{n} f_n(\hat{x})= 0$ for all $n\geq 1$,

$(iii)$ $ d^+ f_0(\hat{x}) (u)+\beta_\infty d^+f_{\infty}(\hat{x})(u)+\sum_{n\geq 1} \beta_{n} d^+f_n(\hat{x})(u)\geq 0$.

$(c)$  there exits a sequence $(\beta_{\infty},\beta_0,\beta_1,...) $ such that:

$(i)$ $\beta_{\infty}\geq 0$, $\beta_n\geq 0$ for all $n\in \N$,

$(ii)$ $\beta_{\infty} f_{\infty}(\hat{x})=0$ and $\beta_{n} f_n(\hat{x})= 0$ for all $n\geq 1$,

$(iii)$ the function $ f_0+\beta_\infty f_{\infty}+\sum_{n\geq 1} \beta_{n} f_n$ has a global minimum on $\Omega$ at $\hat{x}$.

\end{corollary}
\begin{proof} $(a) \Longrightarrow (b)$. Suppose that $(a)$ hold. We apply directely Corollary \ref{opt} using the formulat in $(\ref{star})$ (since $f_n$ is convex for all $n\in \N$), there exits a sequence $(\alpha_{\infty},\alpha_0,\alpha_1,...) $ such that:

$(i)$ $\alpha_{\infty}\geq 0$, $\alpha_n\geq 0$ for all $n\in \N$ and $\alpha_{\infty}+\sum_{n\geq 0} \alpha_n=1$,

$(ii)$ $\alpha_{\infty} f_{\infty}(\hat{x})=0$ and $ \alpha_{n}f_n(\hat{x})= 0$ for all $n\geq 1$,

$(iii)$ $ \alpha_{\infty} d^+f_{\infty}(\hat{x})(u)+\sum_{n\geq 0} \alpha_{n} d^+f_n(\hat{x})(u) \geq 0, \forall u\in E.$

Since  by assumption there exists $w\in E$ such that $ \sup_{n\geq 1} d^+ f_n(\hat{x})(w) < 0$, we have that $\alpha_0\neq0$. Thus, we obtain $(b)$ by setting $\beta_{\infty}:=\frac{\alpha_{\infty}}{\alpha_0}$ and $\beta_n:=\frac{\alpha_{n}}{\alpha_0}$ for all $n\in \N$.

$(b) \Longrightarrow (c)$. This part is given by using the invexity of $(f_n)$ since $f_n$ is convex for all $n\in \N$ (see Example \ref{ho} and the proof of Corollary \ref{cor3}). 

$(c) \Longrightarrow (a)$. This part is clear.
\end{proof}

\section{Examples} \label{SS6}
We end this article with two examples. The aim is to illustrate the use of the upper Dini-derivative in Corollary \ref{pourc}  in the absence of Gateaux-differentiability to find possible solutions to an optimisation problem with a finitely or a countable many inequality constraints.  

We need a few reminders from \cite{Ph}.  We denote $e_k:=(\delta_j^k)$ the elements of $\ell^{\infty}(\N)$ where $\delta_j^k$ is the Kronecker symbol satisfying $\delta_j^k=1$ if $j=k$ and $0$ if $j\neq k$. Let $p: \ell^{\infty}(\N) \longrightarrow \R$ be the function defined for all $x=(x_n)\in \ell^{\infty}(\N)$ by $p(x)=\limsup_n |x_n|.$ It is well-known (see \cite[Example 1.21]{Ph}) that $p$ is a ($1$-Lipschitz) continuous seminorm, nowhere G\^ateaux-differentiable.  On the other hand, it is easily seen that 
$$p(x+u)=p(x), \forall x\in \ell^{\infty}(\N), \forall u\in c_0(\N),$$
where $c_0(\N)$ is the subspace of $\ell^{\infty}(\N)$ consisting on all sequence converging to zero. It follows that  
$$\lim_{t\to 0^+}\frac{p(x + tu)-p(x)}{t}=0,  \forall x\in \ell^{\infty}(\N^*), \forall u \in c_0(\N).$$
Moreover, from \cite[Example 1 part 3)]{BL} since $p$ is convex, we have $D^+ p(x)(u)=d^+ p(x)(u)$ for all $x, u\in \ell^{\infty}(\N)$. Thus, in particular, 
\begin{eqnarray} \label{Ha1}
D^+ p(x)(u)=d^+ p(x)(u)=0, \forall x\in \ell^{\infty}(\N), \forall u \in c_0(\N).
\end{eqnarray}

\begin{example} \label{ex1} Define, for all $x=(x_n)\in \ell^{\infty}(\N)$,  $f_0(x):=p(x)+ h(x),$  where,
$$p(x)=\limsup_n |x_{n}|,$$
$$h(x)=-2\sum_{n\geq 1} 2^{-3n}x_{n}+x_0.$$ 
Consider the following problem, 
\begin{equation*}
(\mathcal{P}_\infty)
\left \{
\begin{array}
[c]{l}
\min f_0(x)\\
x\in B_{\ell^{\infty}(\N)}(0,1),\\
f_n(x):=2^{-n} x^2_{n}+p(x)-x_0 - 2^{-3n}\leq 0, \forall n\geq1.
\end{array}
\right. 
\end{equation*}
We are going to prove  thanks to the necessary and sufficient condition of optimality due to Corollary \ref{pourc} that  $(\mathcal{P}_\infty)$ has a unique solution $\hat{x}$ given by  $\hat{x}_0=0$, $x_{n}=2^{-n}$, for all $n\geq 1$.  However the functions $f_n$ ($n\geq 0$) are nowhere Gateaux-differentiable because of $p$.
\end{example}

\begin{proof} We see easily that the functions $p$ and $f_n$ ($n\geq 0$) are convex Lipchitz on $\Omega$  and that $f_n(x)\to f_{\infty}(x):=p(x)-x_0$ for all $x\in \Omega$ and $\|f_n-f_{\infty}\|_{\textnormal{Lip}}\to 0$, when $n\to +\infty$. We have for each $x\in \Omega$:
\begin{eqnarray} \label{HHa}
d^+ f_0(x)(u)=d^+p(x)(u)-2\sum_{n\geq 1} 2^{-3n}u_{n}+u_0, \forall u \in \ell^{\infty}(\N). 
\end{eqnarray}

\begin{eqnarray} \label{Hha2}
d^+ f_n(x)(u)=d^+p(x)(u)+ 2^{-n+1} x_{n}u_n-u_0, \forall u \in \ell^{\infty}(\N), \forall n\geq 1. 
\end{eqnarray}
 \begin{eqnarray} \label{Hha3}
d^+ f_{\infty}(x)(u)=d^+p(x)(u)- u_0, \forall u \in \ell^{\infty}(\N). 
\end{eqnarray}

 Clearly the sequence $(f_n)$ satisfies the property $(H)$ at every $x\in \Omega$. We observe  using (\ref{Ha1}) and $(\ref{Hha2})$  that $d^+f_n(\hat{x})(e_0)=-1$ for all $n\geq 1$ so by Corollary \ref{pourc},  $\hat{x}$ is a solution of $(\mathcal{P}_\infty)$ if and only if  there exits a sequence $(\beta_{\infty},\beta_0,\beta_1,...) $ such that:

$(i)$ $\beta_{\infty}\geq 0$, $\beta_n\geq 0$ for all $n\in \N$, 

$(ii)$ $\beta_{\infty} f_{\infty}(\hat{x})=0$ and $\beta_{n} f_n(\hat{x})= 0$ for all $n\geq 1$,

$(iii)$ $ d^+ f_0(\hat{x}) (u)+\beta_{\infty} d^+f_{\infty}(\hat{x})(u)+\sum_{n\geq 1} \beta_{n} d^+f_n(\hat{x})(u)\geq 0$.

\vskip5mm
$\bullet$ {\bf Necessary condition:} Suppose that $(i)$, $(ii)$ and $(iii)$ are satisfied. By applying $(iii)$ to $\pm u\in c_0(\N)$, the formulas $(\ref{HHa})$, $(\ref{Hha2})$, $(\ref{Hha3})$ and the fact that $d^+ p(\hat{x})(u)=0$ for all $u\in c_0(\N)$ (see (\ref{Ha1})), we get that 

$$ \left( -2\sum_{n\geq 1} 2^{-3n}u_{n}+u_0\right) -\beta_{\infty} u_0+\sum_{n\geq 1} \beta_{n} \left( 2^{-n+1} x_{n}u_n-u_0\right) = 0, \forall u\in c_0(\N).$$

In particular, we have

$(a)$ $1-\beta_{\infty} -\sum_{n\geq 1}\beta_n =0$ (with $u= e_0$). 

$(b)$ $-2^{-3n+1}+ 2^{-n+1}\beta_n\hat{x}_n =0$ (with $u= e_{n}$), for all $n\geq 1$. 

Necessarily $\beta_n\neq 0$ and $\hat{x}_n=\frac{2^{-2n}}{\beta_n}\geq 0$, by using $(b)$. Using $(ii)$ and the fact that $\beta_n\neq 0$ for all $n\geq 1$, we have $f_n(\hat{x})=0$ for all $n\geq 1$, that is $\hat{x}^2_{n}=2^n(2^{-3n}+\hat{x}_0-p(\hat{x}))=2^{-2n}+2^n(\hat{x}_0-p(\hat{x}))$. Since $(\hat{x}_n)$ is bounded, then we must necessarily have that $p(\hat{x})=\hat{x}_0$, which gives that $\hat{x}^2_{n}=2^{-2n}$ for all $n\geq 1$. Finally, we obtain $\hat{x}_0=p(\hat{x}):=\limsup_n |x_{n}|=0$, $\hat{x}_{n}=2^{-n}$, $\beta_n=2^{-n}$ for all $n\geq 1$ and $\beta_{\infty}=\frac{1}{2}$ (using $(a)$).

$\bullet$ {\bf Sufficient condition :} Suppose that $\hat{x}_0=0$, $\hat{x}_{n}=2^{-n}$, $\beta_n=2^{-n}$ for all $n\geq 1$ and $\beta_{\infty}=\frac{1}{2}$.  We prove that $\hat{x}$ satisfies $(i)$, $(ii)$ and $(iii)$, which shows by Corollary \ref{pourc} that this point is a solution of $ (\mathcal{P}_\infty)$. Clearly $(i)$ and $(ii)$ are satisfied. To see that $(iii)$ is also satisfied, we use the formulas $(\ref{HHa})$, $(\ref{Hha2})$, $(\ref{Hha3})$ and we obtain that for all $u\in \ell^{\infty}(\N)$,
$$d^+ f_0(\hat{x}) (u)+\beta_{\infty} d^+f_{\infty}(\hat{x})(u)+\sum_{n\geq 1} \beta_{n} d^+f_n(\hat{x})(u)=2d^+ p(\hat{x})(u).$$
On the other hand, since $p(\hat{x})=0$ and $p$ is positive, we see from the definition that $d^+ p(\hat{x})(u)\geq 0$ for all $u\in \ell^{\infty}(\N)$. Hence, we obtain $(iii)$. 
\vskip5mm
Finally, we proved thanks to Corollary \ref{pourc} that $\hat{x}$ defined by  $\hat{x}_0=0$, $x_{n}=2^{-n}$, for all $n\geq 1$ is the unique solution of the problem $(\mathcal{P}_\infty)$.
\end{proof}

\begin{example} \label{ex2}

Let $L: \R^{2N+2}\to \R$ be the linear continuous functional defined by $L(x)=\sum_{n=0}^{2N+1}  x_n$, for all $x=(x_0,...,x_{2N+1})\in \R^{2N+2}$. We denote $\|\cdot\|_2$, $\|\cdot\|_1$ and $\|\cdot\|_{\infty}$, the classical $\ell_2$-norm and $\ell_1$-norm of $\R^{N+1}$ respectively and $\|\cdot\|_{\infty}$ the $\ell_\infty$-norm on $\R^{2N+2}$. Set $\Omega:=\{x\in \R^{2N+2}: \|x\|_{\infty}<1\}$. Consider the problem:
\begin{equation*}
(\mathcal{P}_N)
\left \{
\begin{array}
[c]{l}
\min f_0(x) :=\|(x_0,x_2,...,x_{2N})\|_2^2 +\|(x_1,x_3,...,x_{2N+1})\|_1 -L(x)\\
x\in \Omega,\\
f_{n+1}(x):= (n+2)x_{2n}+(n+1) x_{2n+1} +1 \leq 0, \forall 0\leq n\leq N.
\end{array}
\right. 
\end{equation*}
We are going to prove  thanks to the necessary and sufficient condition of optimality due to Corollary \ref{pourc} that  $(\mathcal{P}_N)$ has a unique solution $\hat{x}$ given by  $\hat{x}_{2n}=\frac{-1}{(n+2)}$, $x_{2n+1}=0$, for all $0\leq n \leq N$.  It is clear that $f_n$ is Lipschitz and convex on $\Omega$, for all $0\leq n\leq N+1$. However, because of the norm $\|\cdot\|_1$, we see that $f_0$ is not Gateaux-differentiable at $x$, whenever there exists some $0\leq n_0\leq N$ such that $x_{2n_0+1}=0$. Recall that 
$$d^+ \|\cdot\|_1(0)(u)=\|u\|_1, \forall u\in \R^{N+1}.$$
\end{example}

\begin{proof}  Let $\pi: z\mapsto |z|$, for all $z\in \R$. Clearly,
\begin{eqnarray}\label{pi}
d^+\pi(z)(s)=
\left \{
\begin{array}
[c]{l}
s \textnormal{ if } z>0\\
-s  \textnormal{ if } z<0,\\
|s|  \textnormal{ if } z=0.
\end{array}
\right. 
\end{eqnarray}
Let $(e_n)_{0\leq n\leq 2N+1}$ be the canonical basis of $\R^{2N+2}$. It is easy to see that for every $x\in \R^{2N+2}$ and for every $0\leq n\leq N$, we have 
$$d^+ f_0 (x)(h e_{2n}+ke_{2n+1})=(2x_{2n}-1)h+d^+\pi(x_{2n+1})(k)-k, \forall (h,k)\in \R^2.$$

Consider the stationary sequence $(\widetilde{f}_n)$, where $\widetilde{f}_{n}=f_{n}$ for $0\leq n \leq N$ and $\widetilde{f}_{n}=f_{N+1}$ for all $n\geq N+1$ so that  $\widetilde{f}_\infty=f_{N+1}$.  Clearly, for every $x\in \R^{2N+2}$ we have that $(\widetilde{f}_n)$ satisfies the property $(H)$ at $x$ and for all $n\geq 0$, $d_G \widetilde{f}_{n+1}(x)(w)=\widetilde{f}_{n+1}(w)<0$ with $w=(-1,-1,...,-1)$. 

We apply Corollary \ref{pourc} to the sequence $(\widetilde{f}_n)$. Thus, $\hat{x}\in \R^{2N+2}$ is a solution of $(\mathcal{P}_N)$  if and only if there exits a real numbers $(\beta_0,\beta_1,...\beta_{N+1})$ such that:

$(i)$ $\beta_{n+1}\geq 0$ for all $1\leq n\leq N+1$, 

$(ii)$ $\beta_{n} f_n(\hat{x})= 0$ for all $1\leq n\leq N+1$,

$(iii)$ $ d^+ f_0(\hat{x}) (u)+\sum_{n= 0}^{N} \beta_{n+1} d_G f_{n+1}(\hat{x})(u)\geq 0$.

$\bullet$ {\bf Necessary condition:} Suppose that $(i)$-$(iii)$ are satisfied.  From $(iii)$, we get that for every $0\leq n\leq N$,
$$(2\hat{x}_{2n}-1)h+d^+\pi(\hat{x}_{2n+1})(k)-k+\beta_{n+1}((n+2)h+(n+1)k)\geq 0,   \forall (h,k)\in \R^2.$$
This implies that $(2\hat{x}_{2n}-1)h+(n+2)h\beta_{n+1}\geq 0$ for all $h\in \R$ (with $k=0$) and $d^+\pi(\hat{x}_{2n+1})(k)-k(1-(n+1)\beta_{n+1})\geq 0$, for all $k\in \R$ (with $h=0$). 

From the inequality $(2\hat{x}_{2n}-1)h+(n+2)h\beta_{n+1}\geq 0$ for all $h\in \R$, we get 
$$\hat{x}_{2n}=\frac{1-(n+2)\beta_{n+1}}{2}, \forall 0\leq n\leq N.$$ 
Now, consider the inequality $d^+\pi(\hat{x}_{2n+1})(k)-k(1-(n+1)\beta_{n+1})\geq 0$, for all $k\in \R$. Then, using $(\ref{pi})$:

$\bullet$ If $\hat{x}_{2n+1}>0$, then we get that $k(n+1)\beta_{n+1}\geq 0$ for all $k\in \R$ this implies that $\beta_{n+1}=0$. Thus, $x=\frac{1}{2}$ but in this case we have $f_{n+1}(\hat{x})=(n+2)\hat{x}_{2n}+(n+1) \hat{x}_{2n+1} +1>0$ which is a contradiction. 

$\bullet$ If $\hat{x}_{2n+1}<0$, then we get that $k(-2+(n+1)\beta_{n+1})\geq 0$ for all $k\in \R$ which implies that $\beta_{n+1} =\frac{2}{n+1}>0$ and so $\hat{x}_{2n}=\frac{-(n+3)}{2(n+1)}$. From $(ii)$ we get that $f_{n+1}(\hat{x})=0$, that is, $\hat{x}_{2n+1}=\frac{-(n+2)\hat{x}_{2n}-1}{n+1}=\frac{(n+2)(n+3)-2(n+1)}{2(n+1)^2}\geq0$ which is a contradiction.  

$\bullet$ Hence, $\hat{x}_{2n+1}=0$, so from $(ii)$, we get $\beta_{n+1}((n+2)\hat{x}_{2n}+1)=0$. If $\beta_{n+1}=0$, we have that $\hat{x}_{2n}=\frac{1}{2}$, but in this case $(n+2)\hat{x}_{2n}+(n+1)\hat{x}_{2n+1}+1=\frac{n+2}{2}+1>0$, which is a contradiction. Thus, $\beta_{n+1} \neq 0$ and so $\hat{x}_{2n}=\frac{-1}{n+2}$. 

Thus, we proved that the only candidate to be a solution of $ (\mathcal{P}_N)$ is $\hat{x}$ given by  $\hat{x}_{2n}=\frac{-1}{(n+2)}$, $\hat{x}_{2n+1}=0$, with the Lagrange multipliers $\beta_{n+1}=\frac{1}{n+2}\left(1+\frac{2}{n+2}\right)$, for all $n\in \{0,...,N\}$.

$\bullet$ {\bf Sufficient condition :}  Suppose that $\hat{x}$ is given by  $\hat{x}_{2n}=\frac{-1}{(n+2)}$, $\hat{x}_{2n+1}=0$ and $\beta_{n+1}=\frac{1}{n+2}\left(1+\frac{2}{n+2}\right)$, for all $0\leq n \leq N$. We prove that $\hat{x}$ satisfies $(i)$, $(ii)$ and $(iii)$, which shows by Corollary \ref{pourc} that this point is a solution of $ (\mathcal{P}_N)$. The parts $(i)$ and $(ii)$ are clear. We prove  $(iii)$. Indeed, let $u\in \R^{2N+2}$, we have

\begin{eqnarray*}
d^+ f_0 (\hat{x})(u) &=& -2\sum_{n=0}^{N} \frac{1}{(n+2)} u_{2n} +\sum_{n=0}^N |u_{2n+1}| -\sum_{n=0}^{2N+1} u_n\\
&=&\left( -2\sum_{n=0}^{N} \frac{1}{(n+2)} u_{2n} -\sum_{n=0}^{N} u_{2n} \right) +\left (\sum_{n=0}^N |u_{2n+1}| -\sum_{n=0}^{N} u_{2n+1}\right)\\
&=& -\sum_{n=0}^{N} \left(1+\frac{2}{n+2}\right) u_{2n}+\left (\sum_{n=0}^N |u_{2n+1}| -\sum_{n=0}^{N} u_{2n+1}\right).
\end{eqnarray*}

\begin{eqnarray*}
\sum_{n=0}^{N} \beta_{n+1} d_G f_{n+1}(\hat{x})(u)&= &\sum_{n=0}^{N} \left(1+\frac{2}{n+2}\right) u_{2n}+\sum_{n=0}^{N} \frac{n+1}{n+2}\left(1+\frac{2}{n+2}\right)u_{2n+1}.
\end{eqnarray*}
Thus,
\begin{eqnarray*}
d^+f_0(\hat{x})(u)+\sum_{n= 0}^{N} \beta_{n+1} d_G f_{n+1}(\hat{x})(u) = \sum_{n=0}^{N}\left(|u_{2n+1}|-\left [1-\frac{n+1}{n+2}\left(1+\frac{2}{n+2}\right)\right ]u_{2n+1} \right).
\end{eqnarray*}
Since, $\left |1-\frac{n+1}{n+2}\left(1+\frac{2}{n+2}\right)\right |\leq 1$, for every $n\in \N$, we get that $d^+f_0(\hat{x})(u)+\sum_{n= 0}^{N} \beta_{n+1} d_G f_{n+1}(\hat{x})(u)\geq 0$ for all $u\in \R^{2N+2}$. Hence $(iii)$ is satisfied. 
\vskip5mm
Finally, we proved that  $\hat{x}$ given by  $\hat{x}_{2n}=\frac{-1}{(n+2)}$, $\hat{x}_{2n+1}=0$ for all $0\leq n \leq N$, is the unique solution of $ (\mathcal{P}_N)$.
\end{proof}

\section*{Acknowledgement}
This research has been conducted within the FP2M federation (CNRS FR 2036) and  SAMM Laboratory of the University Paris 1 Panthéon-Sorbonne.

\bibliographystyle{amsplain}

\end{document}